\documentclass{amsart}
\usepackage{amsmath,amsfonts,amsthm}
\usepackage{amssymb,latexsym}
\usepackage{graphics}
\usepackage{rotating}
\usepackage[colorlinks]{hyperref}
\usepackage[all]{xypic}
\usepackage{amssymb}
\usepackage{xspace, enumerate}
\theoremstyle{plain}
\theoremstyle{definition}

\theoremstyle{remark}
\numberwithin{equation}{section}

\theoremstyle{definition}
\newtheorem{dfn}{Definition}
\newtheorem{trm}{Theorem}
\newtheorem{lem}[trm]{Lemma}
\newtheorem{prp}[trm]{Proposition}
\newtheorem{crl}[trm]{Corollary}
\newtheorem{xmp}{Example}
\newtheorem{rmk}{Remark}
\title{Topology of pre-images under Anosov endomorphisms}
\author[{\tiny  Mohammad saeed Azimi  and Khosro Tajbakhsh}]{Mohammad saeed Azimi and Khosro Tajbakhsh}

\address{Mohammad saeed Azimi, Department of Mathematics, Faculty of Mathematical Sciences, Tarbiat Modares University,
Tehran 14115-134, Iran}

\email{saeed.azimi@modares.ac.ir}

\address{Khosro Tajbakhsh, Department of Mathematics, Faculty of Mathematical Sciences, Tarbiat Modares University,
Tehran 14115-134, Iran}

\email{khtajbakhsh@modares.ac.ir , arash@cnu.ac.kr}

\begin{document}

\begin{abstract}
For an endomorphism it is known that if all the points in the manifold have dense sets of pre-images then the dynamical system is transitive. The inverse has been shown for a residual set of points but the the exact inverse has not yet been investigated before. Here we are going to show that under some conditions it is true for Anosov endomorphisms on closed manifolds, by using the fact that Anosov endomorphisms are covering maps.\\
\\

\end{abstract}

\maketitle

\setlength{\parindent}{.5cm}
\noindent
\section{Introduction}
It is well known for non-injective endomorphisms that if for every point the set of pre-images of that point is dense in the manifold then the endomorphism is transitive (i.e. there exists a point that its orbit is dense in the manifold) and in \cite{lizana} Lizana and Pujalz have used this to prove rigidity of transitivity for a special class of endomorphisms on $\mathbb{T}^n$. A very important class of endomorphisms is the class of Anosov endomorphisms. In \cite{LPV}, Lizana, Pinheiro and Varandas have shown that for robustly transitive local diffeomorphisms there is a residual set of point in the manifold such that the points in this set, each one has dense set of pre-images. Here, we use a topological approach, specially the fact that Anosov endomorphisms on a closed manifold are covering maps. We are going to investigate specially about the pre-images of periodic points and show the reciprocative of the well known result above is true for transitive Anosov endomorphisms under some conditions over the geodesics defined by eigenvectors of $Df_x$ for every point. So the set of pre-images of every point is dense in the manifold. Also we will introduce a counterexample for the situation without those conditions.

In this paper, we take all the manifolds to be a closed Riemannian manifold.

Starting from \cite{przytycki} and \cite{manepugh}, the definition of Anosov endomorphism has been an important generalization method of the well known definition of Anosov diffeomorphisms;

\begin{dfn}
Let
$f\in\textup{diff}^r(M)$,
a compact subset $\Lambda\in M$ is called \emph{hyperbolic} with respect to $f$, if for every point $p\in\Lambda$ there is a splitting;
$T_p\Lambda=E^s _p\oplus E^u _p$
and there are
$C>0$ and $0<\lambda<1$
such that
$Df(E^s _p)=E^s _{f(p)}$, $Df(E^u _p)=E^u _{f(p)}$
and for all integer $n\geq 0$;\\
$\forall v\in E^s _p \ \ || Df_p ^n v||\leq C\lambda ^n||v||$,\\
$\forall u\in E^u _p \ \ ||Df_p ^{-n}u||\leq C\lambda ^{-n}||u||$.
\end{dfn}

If $\Lambda =M$ then $f$ is called \emph{Anosov diffeomorphism}.

\begin{xmp}\label{linear diffeo}
\cite{stock} Define $A:\mathbb{T}^2\rightarrow\mathbb{T}^2$ to be;
\[\begin{bmatrix}
2&1\\
1&1
\end{bmatrix} (\textup{mod}1)\]
This is a linear map on $\mathbb{R}^2$ and its eigenvalues are $\frac{3\pm\sqrt{5}}{2}$ which are greater and lesser than one and the eigenspace is the whole $\mathbb{T}^2$
so it is an Anosov diffeomorphism. Also note that $\det A=1$.
\end{xmp}

\begin{rmk}
Considering the map $f:M\rightarrow M$, for every point $x\in M$, the \emph{orbit} of $x$, $O_x$ is $\{f^n(x)|n\in\mathbb{N}\}$. The \emph{trajectory} of $x$, $(x_j)_{j\in\mathbb{Z}}$ such that $x_0=x$, $x_j\in\{f^j(x) \}$ and $f(x_j)=x_{j+1}$. Notice that if $f$ is not injective then $\textup{card}(\{(x_j)_{j\in\mathbb{Z}} \})>1$, but if it is an injective map then the trajectory of each point is unique.  
\end{rmk}

In the case where the map is not injective hyperbolicity is defined considering not just the points but their trajectories under the map.

\begin{dfn}\label{dfn anosov endo}
Let $f:M\rightarrow M$
be a local diffeomorphism, $f$ is called \emph{Anosov endomorphism} if for every trajectory
$(x_n)_{n\in\mathbb{Z}}$ with respect to $f$, for all
$i\in\mathbb{Z}$, $Df(E^s _{x_0})=E^s _{f(x_0)}$, $Df(E^u _{x_i})=E^u _{x_{i+1}}$, $T_{x_i}M=E_{x_i} ^s\oplus E_{x_i} ^u$ and there exist $C>0$ and $0<\lambda<1$ such that;\\
$\forall v\in E^s _{x_i} \ \ ||Df_{x_i} ^n v||\leq C\lambda ^n ||v||$,\\
$\forall u\in E^u _{x_i} \ \ ||Df_{x_i} ^n u||\geq C\lambda ^{-n} ||u|| $. 
\end{dfn}

There is also another way to define Anosov endomorphism;

\begin{dfn}\label{sakai's dfn of anosov endo.}
\cite{sakai} A $C^1$ local diffeomorphism $f:M\rightarrow M$ is called Anosov endomorphism if $Df$ uniformly contracts a continuous sub-bundle $E^s\subset TM$ into itself, and the action of $Df$ on $\frac{TM}{E^s}$ is uniformly expanding.
\end{dfn}

An important result about the definitions above is the continuity of the splitting defined in them ( see \cite{przytycki},\cite{Shub}). 

\begin{xmp}\label{xmp for linear a. endo.}
\cite{przytycki} Define $B:\mathbb{T}^2\rightarrow\mathbb{T}^2$ to be;
\[\begin{bmatrix}
n&1\\
1&1
\end{bmatrix} (\textup{mod}1),\qquad  (n\in\{3,4,5,...\}) \]
The eigenvalues are $\frac{(n+1)\pm\sqrt{(n+1)^2-4(n-1)}}{2}$ and for $n>2$, like the previous example, both of them are greater than zero, one of them is lesser than and the other is greater than one and the eigenspace is the whole manifold so according to definition \ref{dfn anosov endo}, this is an Anosov endomorphism.
\end{xmp}

The main difference between Anosov diffeomorphisms and Anosov endomorphisms comes in to notice in the matter of structural stability. In his thesis Shub claimed that by a procedure similar to the expanding maps, non-injective Anosov endomorphisms are structurally stable. But in \cite{przytycki}, Przytycki proved him wrong, although in the same paper he showed the inverse limit stability of Anosov endomorphisms. Another main difference as it is mentioned above, is the definition of unstable manifolds based on the trajectories so that they can be non-unique \cite{MT}.

An important characteristic of non-injective Anosov endomorphisms is that they are non-trivial covering maps of the manifolds they are defined on \cite{Franks}. In this paper we are going to use this property among other things to show that under a certain condition an Anosov endomorphism is transitive if and only if the set of pre-images of any point is dense in the manifold;

\begin{trm}[Main theorem]\label{theresultiff}
Let $f:M\rightarrow M$ be an Anosov endomorphism and for every point $x\in M$, geodesics defined by eigenvectors of $Df_x$ be dense in $M$ or $f$ is a product of maps with this condition then the set of pre-images of each point, is dense in $M$ if and only if $f$ is transitive.
\end{trm}

\section{Proof of the main theorem}

\begin{dfn}\label{transitivity}
A continuous map $f:M\rightarrow M$ is called \emph{transitive} if for every pair of non-empty open sets $U,V\subset M$, there exists $n\in\mathbb{N}$ such that $f^n(U)\cap V\neq\phi$. 
\end{dfn}

There is this well known proposition about transitivity;

\begin{prp}\label{trnstvtequal} (\cite{stock}, proposition 2.2.1)
Let $M$ be a complete space without any isolated point and $f:M\rightarrow M$,continuous, $f$ is transitive, if and only if there exists $p\in M$ such that $\overline{O_p}=M$.
\end{prp}

In the context of dynamical systems, because of the manifolds they take in to account, the proposition above is often considered as the definition of transitivity.

Another well known result in the matter of transitivity is about hyperbolic linear automorphisms;

\begin{prp}
\cite{stock} Let $A:\mathbb{T}^2\rightarrow\mathbb{T}^2$ be a hyperbolic linear toral automorphism, $A$ is transitive.
\end{prp}

If $f$ is a diffeomorphism then definition \ref{transitivity} is also true for $f^{-1}$, so in that, the set $\mathbb{N}$ can be changed to $\{-1,-2,-3,...\}$ and the definition remains intact. But in the case of Anosov endomorphisms, $f^{-1}$ is meaningless but we can still investigate the set of pre-images of the points in $M$ under Anosov endomorphisms.

In the following, we also need these two definitions;

\begin{dfn}\label{index}
Let $f:M\rightarrow M$ be a transitive Anosov endomorphism, for each point $x\in M$ we call the $\textup{dim}(E^s _x)$, \emph{index} of $f$. 
\end{dfn}

Note that because of the continuity of the splitting in the definition of Anosov endomorphisms and because the map is transitive, index of $f$ does not change by points.

\begin{dfn}\label{dfn of degree}
Let $f:M\rightarrow M$ be an Anosov endomorphism with $n$ number of pre-images for each point in $M$ ($n$ number of sheets for the covering it makes), we call $n$, the \emph{degree} of an Anosov endomorphism $f$.
\end{dfn}

\begin{rmk}
For Anosov maps the degree is the same for every point because if there are points with different degree then the map has singularities in some points which is not possible for Anosov maps. Although by simple modifications, we can also deduce the result of this paper to the maps that have finitely many degrees over the manifold. 
\end{rmk}

\begin{rmk}\label{sheets of a cover}
Anosov endomorphisms are  covering maps and except for Anosov diffeomorphisms, they are not trivial and the manifolds on which they are defined, are evenly covered \cite{Franks}. Since in this paper we take the manifold $M$ to be a closed, there are a finite number of sheets for these covering maps and the number equals the degree of the Anosov endomorphism.

Also the determinant of Jacobian of the Anosov endomorphism equals the degree of the map \cite{przytycki}.
\end{rmk}

Let $f:M\rightarrow M$ be a transitive Anosov endomorphism, $(f,M)$ is a cover for $M$. Considering the endomorphism $f$, because $M$ is compact there is a finite number of sheets (equal to the degree of $f$), $S(1),S(2),S(3),...,S(k)\subset M$, each of them homeomorphic to $M$ under $f|_{S(i)}:S(i)\rightarrow M$ and for every point $x\in M$ there is a $d_1>0$ such that if $i\neq j$, $\textup{d}(x(i),x(j))>d_1$ for all $x(i)$ and $x(j)$ in $f^{-1}(x)$ and uniquely in $S(i)$ and $S(j)$. Also for every $1\leq j\leq k$, $S(i,j):=(f|_{S(i)})^{-1}(S(j))\subset S(i)$ and $f^2|_{S(i,j)}\rightarrow M$ is a homeomorphism. This also means that the interior of $S(i,j)$ is not empty and $\textup{diam}(S(i,j))>0$ for all $i$ and $j$. So $(f^2,M)$ is a cover for the manifold with exactly, $k^2$ sheets such that there are $k$ sheets as subsets of each $S(i)$, we denote them by $S(i,1),S(i,2),\dots,S(i,k)\subset S(i)$ and each of them is homeomorphic to $M$ by $f^2$. So considering all $S(i)$s, there are $k^2$ sets $S(i_1,i_2)\subset M$. By induction, for every $n\in\mathbb{N}$, $(f^n,M)$ is a cover for $M$ with $k^n$ sheets. Also $M$ is evenly covered and $S(i_1,i_2,i_3,\dots,i_n)$s do not intersect. For every sheet $S(i_1,\dots,i_n)$, the map $f^{n}|_{S(i_1,\dots,i_n)}:S(i_1,\dots,i_n)\rightarrow M$ is a homeomorphism and there is $d_n>0$ such that for every pair of the $n$th pre-images of $x$, $x(i_1,\dots,i_n)$ and $x(j_1,\dots,j_n)$ respectively in sheets $S(i_1,\dots,i_n)$ and $S(j_1,\dots,j_n)$, $\textup{d}(x(i_1,\dots,i_n),x(j_1,\dots,j_n))>d_n$. 

We saw that $S(i_1,\dots,i_{n-1},i_n)$s are subsets of $S(i_1,\dots,i_{n-1})$ and following this, step by step, finally they are subsets of $S(i_1)$. In every sheet of $(f^r,M)$ there are $k$ sheets of $(f^{r+1},M)$ and $d_{r+1}=\frac{d_r}{k}$ so  $d_{r+1}=\frac{d_1}{k^{r+1}}$ and so on.

About the distribution of the sheets of the covers $(f^n,M)$s, by the context above we have: For all open sets $U\subset M$ and for all $n\in\mathbb{N}$, there is a sheet $S(i_1,\dots,i_n)$ of the cover $(f^n,M)$ such that $S(i_1,\dots,i_n)\cap U\neq\emptyset$. As $l\in\mathbb{N}$ gets greater, if we consider the sheets $S(i_1,\dots,i_n,\dots,i_{n+l})\subset S(i_1,\dots,i_n)$ of $(f^{n+l},M)$ then there exists $N\in\mathbb{N}$ such that for all $m>N$ there is $S(i_1,\dots,i_m)\cap U\neq\emptyset$ and for all $l\in\mathbb{N}$, $S(i_1,\dots,i_m,\dots,i_{m+l})\cap U\neq\emptyset$. If $f$ was an expanding map then the intersection of the sheets with $U$ would be inclusion which would give the density of pre-images of every point (See proposition \ref{density in expandings}).

Similar to the diffeomorphism case we have the two following propositions;

\begin{prp}\label{backtrans}
Let $M$ be a compact metric space and $f:M\rightarrow M$ be an endomorphism. If $f$ is transitive then for every pair of non-empty open sets $U$ and $V$ in $M$, there is $n\in\mathbb{N}$ such that $f^{-n}(U)\cap V\neq\phi$.
\end{prp}

\begin{proof}
Suppose $U$ and $V$ to be open sets in $M$ and $k$ be the degree of $f$. There is $n\in\mathbb{N}$ such that we have $f^n (U)\cap V\neq\emptyset$ then $f^{-n}(f^n (U)\cap V)\neq\emptyset$; but $f^{-n}(f^n (U)\cap V)=f^{-n}(f^n(U))\cap f^{-n}(V)$ and $f^{-n}(f^n(U))$ is the union of the sets $U(i_1,\dots,i_n)=f^{-n}(f^n(U))\cap S(i_1,\dots,i_n)$. There is $U(i_1,\dots,i_n)=U$ and because $f^n$ is a covering map, each one of $U(i_1,\dots,i_n)$s is homeomorphic to $U$ and $U(i_1,\dots,i_n)\cap f^{-n}(V)\neq\emptyset$ for all $(i_1,\dots,i_n)$ ($i_r\in\{1,\dots,k\}$). Hence $U\cap f^{-n}(V)\neq\emptyset$.

\end{proof}

The following proposition is a crucial fact about Anosov endomorphisms;

\begin{prp}\label{clperf=omegaf}
(\cite{przytycki}, proposition 3.2) Let $f:M\rightarrow M$ be an Anosov endomorphism then $\overline{per(f)}=\Omega(f)$.
\end{prp}

Now we want to see if there is a point which its set of pre-images is dense in the manifold, first we have this rather obvious result;

\begin{lem}\label{densityofpreimagesofadense}
Let $f:M\rightarrow M$ be a transitive Anosov endomorphism; if a set is dense in $M$ then also the set of its pre-images is dense in $M$.
\end{lem}
\begin{proof}
$f$ is an Anosov endomorphism so, as we mentioned above $f^n$ is a covering map for $M$ for every $n\in\mathbb{N}$ therefore each sheet of every cover $(f^n,M)$ for $M$, is homeomorphic to $M$ so if a set is dense in $M$ then its pre-image in each sheet of the cover is dense in that sheet. $M$ is the union of the sheets of a cover $(f^n,M)$. Thus the set containing union of the pre-images of a dense set of $M$ is dense in $M$.
\end{proof}

It implies that the points which have dense orbits have dense sets of pre-images;

\begin{prp}\label{densityofpreimagesofadenseorbit}
Let $M$ be a closed manifold and $f:M\rightarrow M$ be an Anosov endomorphism then every point with a dense orbit, has a dense set of pre-images.
\end{prp}

\begin{proof}
Suppose that $p\in M$ is a point with dense orbit. For each $\epsilon>0$ there exists $n\in\mathbb{N}$ such that $\{f(p),f^2(p),\dots,f^n(p)\}$ is $\epsilon$-dense in $M$. In every sheet $S(i_1,i_2,\dots,i_n)\subset M$, of the cover $(f^n,M)$ the subset of pre-images of the point $p$, $(f^n |_{S(i_1,i_2,\dots,i_n)})^{-1}(\{f(p),f^2(p),\dots,f^n(p)\})$, is homeomorphic to $\{p,f(p),f^2(p),\dots,f^{n-1}(p)\}$ under $f^n:S(i_1,i_2,\dots,i_n)\rightarrow M$, and it is $\epsilon$-dense in $\mathbb M$. Because $\epsilon$ and also $S(i_1,i_2,\dots,i_n)$ are chosen arbitrarily, by Lemma \ref{densityofpreimagesofadense}, the set of the pre-images of $p$ is dense in $M$.
\end{proof}

Notice that because a linear Anosov endomorphism is transitive and there is a large set of points with dense orbit under it in $\mathbb{T}^n$, the Lemma and proposition above are true for such systems. Specially because the points with dense orbit are dense in $\mathbb{T}^n$, Lemma \ref{densityofpreimagesofadenseorbit} shows that the set of the points with dense set of pre-images is at least dense in $\mathbb{T}^n$. We are going to investigate this more precisely on closed manifolds.

By modifying an important results about Anosov diffeomorphisms,\cite{W} we have;

\begin{prp}\label{density of a dense set}
The set of points with dense set of pre-images under a transitive Anosov endomorphism, is at least a dense set in $M$.
\end{prp}

\begin{proof}
For every $\epsilon>0$ there exists a finite basis $\beta _\epsilon=\{B_1,B_2,\dots, B_n\}$ for $M$ consisting of $\epsilon$-discs. Denote $\cup _{i=1} ^{\infty} f^i (B_j)$ by $E_j$. Because $f$ is an Anosov endomorphism, it is an open map and because $f$ is also transitive, $E_j$ is open and dense. $M$ is a Bair space so $\cap _{j=1} ^n E_j\neq\emptyset$ and there exists a point $p\in\cap _{j=1} ^n E_j$ then for every $1\leq j\leq n$ there is $i\in\mathbb{N}$ such that $p\in f^i(B_j)$. So $f^{-i}(p)\cap B_j\neq\emptyset$. Because it is true for all $\epsilon>0$ and all the points in $\cap _{j=1} ^n E_j$, the set of the points with dense set of pre-images is dense in $M$.
\end{proof}

\begin{trm}
\cite{przytycki} Let $f:M\rightarrow M$ be an Anosov endomorphism then there is $\epsilon$ such that for any trajectory $(x_i)_{i\in\mathbb{Z}}$ of any $x\in M$ the set;
\[W^s _{x_i,\epsilon}=\{y\in M|\forall n\in\mathbb{N}\quad \textup{d}(f^n(y),f^n(x_i))<\epsilon \} \]
is a manifold which is called local stable manifold of $x$, and the set;
\[W^u _{x_i,\epsilon}=\{y\in M|\exists (y_n)_{-\infty} ^0\quad \forall n\in\mathbb{N}\quad \textup{d}(y_{-n},x_{i-n})<\epsilon \} \]
is a manifold which is called local unstable manifold of $x_i$ related to the trajectory $(x_i)_{i\in\mathbb{Z}}$ under $f$.  
\end{trm}

Following the theorem above we have;

\begin{dfn}
The sets;
\[W^s _x=\cup _{i=0} ^{\infty}f^{-n}(W^s _{f^n(x),\epsilon}) \]
and
\[W^u _x=\cup _{i=0} ^{\infty}f^{n}(W^u _{x_{-n},\epsilon}) \]
respectively are called the stable and unstable sets of the point $x\in M$.
\end{dfn}

Notice that $W^{s(u)} _x=\{y\in M|\textup{d}(f^n(y),f^n(x))\rightarrow 0(\infty) \}$. Also note that the stable and unstable sets defined above may not even be submanifolds if the degree of $f$ is greater than one.

If $f:M\rightarrow M$ be a transitive diffeomorphism then the stable and unstable manifolds of every points are dense in $M$ \cite{stock}. An essential concept that make this happen, is \emph{local product structure} of $M$ under $f$ \cite{Shub}. An endomorphism is locally diffeomorphism so by indicating $\tau$ such that $W^u _\tau$ be unique for each point, and modifying the definition for the Anosov-endomorphisms case we have;

\begin{dfn}\label{dfn lps for endos}
A closed hyperbolic invariant set is said to have a \emph{local product structure} if for small $\epsilon<\tau$ and $\delta$, $W^u _{\epsilon,x}\cap W^s _{\epsilon,y}$ is unique and belongs to the hyperbolic set whenever $\textup{d}(x,y)<\delta$.
\end{dfn}

Also in \cite{przytycki}, Przytcky has shown this in the inverse limit space. Therefore exactly the same as the diffeomorphism case,\cite{Shub} we have;

\begin{prp}\label{prp lps for endos}
Let $f:M\rightarrow M$ be a hyperbolic endomorphism, if $\overline{Per(f)}$ is hyperbolic then it has a local product structure.
\end{prp}

The maps we are studying are Anosov and by proposition \ref{clperf=omegaf}, the set of periodic points is dense in $M$ so the whole manifold has a local product structure under $f$ and modifying proposition 5.10.3 of \cite{stock} we have;

\begin{prp}\label{densityofmanifolds}
Let $f:M\rightarrow M$ be an Anosov endomorphism and $\Omega(f)=M$ then the pre-images of stable and unstable sets are dense in $M$.
\end{prp}

\begin{proof}
With an argument like the diffeomorphism case, the unstable manifold of a point, is dense in $M$ also Przytcky in \cite{przytycki}, has proved this by lifting $f$ to inverse limit space. So by the proposition \ref{densityofpreimagesofadense} its set of pre-images is dense in $M$. We show that the set of pre-images of a stable manifold of every point is dense. By proposition \ref{clperf=omegaf}, the set of periodic points under $f$, is dense in $M$ so it is $\epsilon$-dense in every sheet of each one of the covers $(f^n,M)$, for every $n\in\mathbb{N}$. Suppose that $\epsilon$ is chosen such that there exists $\delta>0$, if $\textup{d}(x,y)<\delta$ ($x,y\in M$), then for each trajectory $(y_i)_{i\in\mathbb{Z}}$, $W^s _{\epsilon,x}\cap W^u _{\epsilon, (y_i)}$ contains exactly one point and following the statements before the proposition, if $\epsilon$ is small enough, it meets the conditions of local product structure definition. Now consider $B:=\{p_i\in Per(f)|i=1,2,\dots,N\}$ to be an $\frac{\epsilon}{4}$-dense set in $M$ so that local unstable manifold of each point in $B$ transversally intersects with local stable manifold of the points in $B$ $\epsilon$-close to it.

Suppose that $\tau\in\mathbb{N}$ the product of the periods of all the points in $B$ and put $g=f^\tau$. Suppose that $S(j_1,j_2,\dots,j_r)$ is a sheet of the cover $(f^r,M)$ (let $\textup{deg}(f)=k$) and $\{p_i (j_1,\dots,j_r)|i=1,2,\dots,N\}$ is the pre-image of $B$ in $S(j_1,\dots,j_r)$, under $g$. Let $W^s _x (j_1,\dots,j_r)$ be the pre-image of $W^s _x$ for every $x\in M$, in $S(j_1\dots,j_r)$. We have;

\begin{lem}
With the assumptions above, if $\textup{d}(W^s _y (j_1,\dots,j_r)),p_i)<\frac{\epsilon}{2}$ and $\textup{d}(p_i,p_l)<\frac{\epsilon}{2}$ then there are $m\in\mathbb{N}$ and $S(j_1,\dots,j_r,\dots,j_{r+l})$, a sheet of the cover $(g^m,M)$ and a subset of $S(j_1,\dots,j_r)$, such that;
\[\textup{d}(g^{-m}(W^s _y (j_1,\dots,j_r,j_{r+1})),p_i (j_1,\dots,j_r,j_{r+1}))<\frac{\epsilon}{2} \]
and
\[\textup{d}(g^{-m}(W^s _y (j_1,\dots,j_r,j_{r+1})),p_l (j_1,\dots,j_r,j_{r+1}))<\frac{\epsilon}{2}.\]
\end{lem}

\begin{proof}
There exists $z\in W^s _y (j_1,\dots,j_r)\cap W^u _{\frac{\epsilon}{2},p_i}(j_1,\dots,j_r)$ so there is a $t_0\in\mathbb{N}$ such that $\textup{d}(g^t (z),p_i)<\frac{\epsilon}{2}$ for every $t>t_0$. So $\textup{d}(g^{-t} (z),p_l)<\epsilon$. Therefore like in the previous step there exists a point $w\in W^s _{g^t (z)}(j_1,\dots,j_r,\dots,j_t)\cap W^u _{\frac{\epsilon}{2},p_l}(j_1,\dots,j_r,\dots,j_t)$. Hence there is a $b_0\in\mathbb{N}$ such that $g^{-b}(w)\in S(j_1,\dots,j_r,\dots,j_t,\dots,j_b)$ and $\textup{d}(g^{-b} (w),p_l)<\frac{\epsilon}{2}$ for every $b>b_0$. Taking $S(j_1,\dots,j_r,\dots,j_{r+l})=S(j_1,\dots,j_r,\dots,j_{r+t},\dots,j_b)$ and $m=b_0 +t_0$, the proof completes.
\end{proof}

Since $M$ is compact and connected, any two periodic points $p_1$ and $p_2$ can be connected together by a path containing not more than $N$ periodic points with less than $\frac{\epsilon}{2}$ distance between any two consecutive periodic points. By the Lemma above, for any $x\in M$ and $\epsilon>0$ $g^{-Nm}(W^s _x)$ is $\epsilon$-dense in a sheet of the cover $(f^{Nm\tau},M)$ and a subset of $S(j_1,\dots,j_r)$. Because it is correct for every $\epsilon$ and the sheet $S(j_1,\dots,j_r)$ is chosen arbitrarily, the proposition follows.
\end{proof}

We saw that the set of pre-images of a point with dense forward orbit under a linear Anosov endomorphism $A:\mathbb{T}^n\rightarrow\mathbb{T}^n$ which is not an expanding map, is dense in $\mathbb{T}^n$. About Anosov diffeomorphisms, this is it but for expanding maps we have this well known result;

\begin{prp}\label{density in expandings}
Let $f:M\rightarrow M$ be an expanding map, each point in $M$ have a dense set of pre-images in $M$.
\end{prp}

\begin{proof}
Suppose that $D_\epsilon$ is an $\epsilon$-disk in $M$, for every $\epsilon>0$. Since $f$ is an expanding map, there exists $H\subset D_\epsilon$ and $n\in\mathbb{N}$ such that $f^n(H)=M$. Therefore for every $p\in M$ there is $x\in f^{-n}(p)\cap D_\epsilon$.
\end{proof}

For Anosov endomorphisms which are not diffeomorphisms or expanding maps, it is different from diffeomorphisms because they are non trivial covering maps also it is different from expanding maps because they also have a contracting factor. Therefor in addition to the points with dense orbit we are going to investigate about pre-images of the points that their orbits and hence their $\omega$-limit sets have various topological properties.

Notice that each periodic point under an Anosov endomorphism is also an image of a non-periodic point. It is because of the degree of the Anosov endomorphism being greater than 1 and pre-image set of any point contains more than one point but at most one of them is periodic. So we have;

\begin{prp}\label{periodicsaredense}
Let $f:M\rightarrow M$ be an Anosov endomorphism then the set of pre-images of the set of all the periodic points, $\cup_{n\in\mathbb{N}} f^{-n}(Per(f))$ such that $f^{-i}(Per(f))=\cup_{x\in Per(f)}f^{-i}(x)$, is dense in $M$.
\end{prp}

\begin{proof}
Because $f$ is an Anosov endomorphism on a closed manifold $M$, we have $\overline{Per(f)}=M$. So by Lemma \ref{densityofpreimagesofadense} the set containing all the pre-images of all the periodic points is dense in $M$.
\end{proof}

Now we investigate the pre-images of an arbitrary periodic point under transitive Anosov endomorphisms. First there are some examples that the pre-images of at least some of the fixed points are not dense in the manifold;

\begin{xmp}\label{conter xmp}
Define $B:\mathbb{T}^3\rightarrow\mathbb{T}^3$ to be;
\[\begin{bmatrix}
2&0&0\\
0&2&1\\
0&1&1
\end{bmatrix} (\textup{mod}1).\]
The eigenvalues are $2$ and $\frac{3\pm\sqrt{5}}{2}$ and it is a transitive Anosov endomorphism defined by the product of doubling map over $S^1$ and $A$ in the example \ref{linear diffeo} over $\mathbb{T}^2$. For any point in $\mathbb{T}^3$, the geodesic defined by the eigenvalue $2$'s eigenvector is $S^1$ which is not dense in $M$. obviously the set of pre-images of the fixed point $(0,0,0)\in\mathbb{T}^3$ is dense in $S^1\times(0,0)$ and is not dense in $M$. This can also be stated by this;
\end{xmp}

\begin{rmk}\label{reducibility}
If there exists a factor (A semi conjugate map) $f:N\rightarrow N$ for the Anosov endomorphism $F:M\rightarrow M$ such that $f$ is an Anosov endomorphism or an expanding map and the projection of $F$ on $\frac{M}{N}$ is an Anosov diffeomorphism then the pre-images of any point $p\in M$, are distributed in $\{p\}\times N$ and if $p$ is fixed under $F$ then, like the example above, its set of pre-images is not dense in the whole manifold.

This condition in the study of rigidities in Anosov group actions is commonly called \emph{reducibility} \cite{Spatzier}. 
\end{rmk}

So it is possible for the pre-images of a fixed point (or periodic point) under an Anosov endomorphism to be dense in a non trivial subset of $M$ but in many cases they are dense in $M$;

\begin{trm}\label{maintheoremforpriodics}
Let $f:M\rightarrow M$ be a transitive Anosov endomorphism such that for every point $x\in M$ geodesics defined by eigenvectors of $Df_x$ are dense in $M$, then the periodic points have dense sets of pre-images under $f$.
\end{trm}

\begin{proof}
Without any loss of generality, let $p\in M$ be a fixed point and $k$ to be the degree of $f$ and $\{x(i)\in S(i)|i\in\{1,\dots,k\}\}=f^{-1}(p)$. Then define for each $x(i)\in f^{-1}$;
\[\alpha_{-1}(x(i)):=\min_{x(j)\in f^{-1}(p)}(\textup{d}(x(i),x(j))) \]
and;
\[\beta_{-1}:=\max_{x(i)\in f^{-1}(p)}(\alpha_{-1}(x(i))). \]
$\beta_{-1}$ is the maximum distance possible between a point in $f^{-1}(p)$ and its nearest point in $f^{-1}(p)$ that is not equal to the first one. Then for every $n\in\mathbb{N}$ define $\beta_{-n}$ for the points in $f^{-n}(p)$ in the same way. $(f^n,M)$s are covers for $M$ and since $M$ is a closed manifold, as $n$ gets bigger the volume of each sheet of the cover gets smaller accordingly and so the distance between the pre-images of each point gets smaller (See remark \ref{sheets of a cover}). Hence for every point $x(i_1,\dots,i_n)$ in $f^{-n}(p)=\{x(i_1,\dots,i_n)\in S(i_1,\dots,i_n)|i_1,\dots,i_n\in\{1,\dots,k\} \}$;
\[\alpha_{-n}(x(i_1,\dots,i_n))=\min_{x(j_1,\dots,j_n)\in f^{-n}(p)}(\textup{d}(x(i_1,\dots,i_n),x(j_1,\dots,j_n)))\leq\frac{\alpha_{-1}(x(i_1))}{k^n} \]
and similarly;
\[\beta_{-n}=\max_{x(i_1,\dots,i_n)\in f^{-n}(p)}(\alpha_{-n}(x(i_1,\dots,i_n)))\leq\frac{\beta_{-1}}{k^n}. \]
This means that for every $\epsilon>0$ there is $n\in\mathbb{N}$ such that $\beta_{-n}<\epsilon$.

Now connect each pair of points $x(i_1,\dots,i_n)$ and its nearest points in $\{f^{-n}(p)\}$ with geodesics by the length $\alpha_{-n}(x(i_1,\dots,i_n))$, between them. By this procedure we will have a subset of $M$ consisting of some connected components $c^n _1,\dots,c^n _{t_n}$ ($t_n\in\mathbb{N}$), in which $f^{-n}(p)$ is $\beta_{-n}$-dense. These components are disconnected because for every point in them there are other points in the same connected component with less distance than the points in other components. Now connect the components by a geodesic $c^n(i,j)$ such that $i,j\in\{1,\dots,t_n\}$, from the two points $x(i_1,\dots,i_n)\in c^n _i$ and $x(j_1,\dots,j_n)\in c^n _j$ that have the least distance. We call this set $\xi^{-n}(p)$. For every $\epsilon$ there is $m>n$ and the cover $(f^m,M)$ such that for $x(i_1,\dots,i_n)\in c^n _i$ and $x(j_1,\dots,j_n)\in c^n _j$, $\textup{d}(f^{-m}(x(i_1,\dots,i_n)),f^{-m}(x(j_1,\dots,j_n)))<\epsilon$. Now by connecting the points in $f^{-m}(p)$ by geodesics and repeating the process above, we have $\xi^{-m}(p)$ in which $f^{-m}(p)$ is $\epsilon$-dense. Thus for every $\epsilon$ there is $m\in\mathbb{N}$ and $\xi^{-m}(p)\subset M$ such that $f^{-m}(p)$ is an $\epsilon$-dense subset in $\xi^{-m}(p)$. As $m$ goes to infinity there is a subset of $M$ in which $\lim_{m\rightarrow\infty}f^{-m}(p)$ is dense;

\begin{lem}
There exists the set $\xi_{p}:=\lim_{m\rightarrow\infty}\xi^{-m}(p)$ in which $\lim_{m\rightarrow\infty}f^{-m}(p)$ is dense.
\end{lem}

\begin{proof}
Suppose the opposite, so there is $N\in\mathbb{N}$ and $\epsilon$ such that for all $n>N$ there exists a point $x\in\xi^{-n}(p)$ such that $\textup{d}(x,f^{-n}(p))>\epsilon$. Then by considering the definition of $\xi^{-n}(p)$s, for all $n>N$ there are $x(i_1,\dots,i_n)$ and $x(j_1,\dots,j_n)$ in $f^{-n}(p)$ so that for every two trajectories $(x_{-m})_{m\in\mathbb{N}}$ and $(y_{-m})_{m\in\mathbb{N}}$ such that $x_{-j}$ and $y_{-j}$ respectively are in $f^{-m}(x(i_1,\dots,i_n))$ and $f^{-m}(x(i_1,\dots,i_n))$;
\[\lim_{m\rightarrow\infty}\textup{d}((x_{-m})_{m\in\mathbb{N}},(y_{-m})_{m\in\mathbb{N}})\neq0\]
which by remark \ref{sheets of a cover} is a contradiction with the definition of $x(i_1,\dots,i_n)$s.
\end{proof}

Now let $D_\delta(p)$ be a $\delta$-disc around $p$ where $\delta<\tau$ in definition \ref{dfn lps for endos} and let $V_i$s be the eigenvectors of $df_p$ and $W^i _p$s be the geodesics defined by $V_i$s. Also $W^i _{p,\delta}:=W^i _p\cap D_\delta (p)$. Let $(x_{j})$ be a trajectory of $p$, in the pre-images of $f$, the contraction is on the $W^i _{x_{j}}$s where $W^i _{p,\delta}\subset W^u _{p,\delta}$ and the sheets of the covers $(f^n,M)$ are made because of that contraction. Also $f^{-n}(p)\subset f^{-n}(W^u _{p,\delta})$. So if there exists $r_1,r_2,\dots,r_l$ such that $W^i _{p,\delta}\subset W^u _{p,\delta}$, $i=r_1,\dots,r_l$, and $W^i _p$s are not dense in $M$ then there exists a nowhere-dense subset $L$ in $M$, that is defined in each point $y\in L$ by parallel translation of $V^i _p$s then $f$ is an expansion on $L$ and for all $n\in\mathbb{N}$, $f^{-n}(L)\subset L$ so if $p\in L$ then all the pre-images of $p$ remain in $L$ hence they are not dense in $M$. So those $W^i _p$s that $W^i _{p,\delta}\subset W^u _{p,\delta}$
have to be dense in $M$.

The procedure above is the geometric counterpart of irreducibility because following that there cannot be any non-trivial endomorphism factor for the Anosov endomorphism (See remark \ref{reducibility} and example \ref{conter xmp}). 

Now following the proof, let $\xi_{p,\delta}:=\xi_p\cap D_\delta(p)$ and suppose that each $W^i _p$ such that $W^i _{p,\delta}\subset W^s _{p,\delta} $, is dense in $M$ (notice that by this, also $W^i _x$ for every $x\in\cup_{n\in\mathbb{N}}f^{-n}(p)$, is dense in $M$), and define $\xi_{p,\delta} ^i$s by canonical projections $\pi_i(\xi_{p,\delta})\rightarrow W^i _p$ in which $W^i _{p,\delta}\subset W^s _{p,\delta}$. Then for every $W^i _p$, $\lim_{m\rightarrow\infty}f^{-m}(\xi_{p,\delta} ^i)$ is dense in it and since each $W^i _p$ is dense in $M$, $\lim_{m\rightarrow\infty}f^{-n}(\xi_{p,\delta} ^i)$ is dense in $M$. Thus $\lim_{m\rightarrow\infty}f^{-m}(\xi_{p})$ is dense in $M$ and therefore the set of pre-images of $p$ is dense in $M$.
\end{proof}

Due to the linear Anosov endomorphisms being transitive, the proposition above gives us;

\begin{crl}\label{preimagesoflinearperiodics}
Let $A:\mathbb{T}^n\rightarrow\mathbb{T}^n$ be a linear Anosov endomorphism where $A$ is an Anosov endomorphism of degree greater than one and eigenvectors of $A$ define dense geodesics in $\mathbb{T}^n$ then the set of pre-images of a periodic point is dense in $\mathbb{T}^n$.
\end{crl}

\begin{prp}\label{resultforws}
Let $f:M\rightarrow M$ be an Anosov endomorphism, if the set of pre-images of a point $p\in M$ under $f$, is dense in $M$ then the points in $W^s (p)$ and $W^u (p)$ have dense sets of pre-images under $f$.
\end{prp}

\begin{proof}
For all $x\in W^s _p$ ($p\in M$), $O_p\in\omega (x)$. So if $O_p$ has a dense set of pre-images in $M$ then the set of pre-images of $\omega(x)$ is dense in $M$. Since $\omega(x)$ and following that $O_x$ are dense in $M$, then by lemma \ref{densityofpreimagesofadense}, the pre-images of $O_x$ is dense in $M$. Hence the set of pre-images of $x$ is dense in the manifold. If $x\in W^u _p$, $O_p\in\alpha (x)$ and clearly if $O_p$ is dense or its set of pre-images is dense in $M$ then $x$ has a dense set of pre-images.
\end{proof}

\begin{prp}\label{fornonperiodicnongeneric}
Let $f:M\rightarrow M$ be a transitive Anosov endomorphism of degree greater than one and for every point $x\in M$ geodesics defined by eigenvectors of $Df_x$ are dense in $M$. Every point which is not periodic or its $\omega$-limit set does not have a dense set of pre-images in $M$, has a dense set of pre-images.
\end{prp}

\begin{proof}
Suppose that $x\in M$ is a non-periodic point that also does not have a dense orbit. For these points we consider $\omega (x)$. If $int(\omega (x))\neq\emptyset$ then by proposition \ref{backtrans}, $\cup_{n\in\mathbb{N}}f^{-n}(\omega(x))$ is dense in $M$ and by proposition \ref{resultforws} the set of pre-images of $x$ is dense in $M$. If $int(\omega (x))=\emptyset$, by a procedure like in the Proof of Theorem \ref{maintheoremforpriodics} and considering the pre-images of $\omega(x)$ instead of the pre-images of a fixed point, and again by proposition \ref{resultforws} the set of pre-images of $x$ is dense in $M$.
\end{proof}

To sum it up, by propositions \ref{densityofpreimagesofadenseorbit}, \ref{resultforws} and \ref{fornonperiodicnongeneric} and theorem \ref{maintheoremforpriodics} we have;

\begin{trm}\label{resultfortransitiveendos}
Let $f:M\rightarrow M$ be a transitive Anosov endomorphism of degree greater than one and for every point $x\in M$ geodesics defined by eigenvectors of $Df_x$ are dense in $M$, then for every point, the set of pre-images is dense in the manifold.
\end{trm}

According to this and the proof of theorem \ref{maintheoremforpriodics}, for product manifolds we also have;

\begin{crl}\label{product of them}
Let $f:M\rightarrow M$ and $g:N\rightarrow N$ be transitive Anosov endomorphisms such that the pre-images of any point in $M$ and any point in $N$ respectively under $f$ and $g$ are dense in $M$ and $N$, then the pre-images of any point in $M\times N$ under $(f,g)$ is dense in $M\times N$. 
\end{crl}

\begin{proof}
For every open set $U\subset M$, $V\subset N$ and for every $(p,q)\in M\times N$, since $\cup_{n\in\mathbb{N}}f^{-n}(p)$ is dense in $M$, there is $y_1\in(\cup_{n\in\mathbb{N}}f^{-n}(p))\cap U$ and in the same way, for $(y_1,q)\in\{y_1\}\times N$ there is $(y_1,y_2)\in(\cup_{n\in\mathbb{N}}(y_1,g^{-n}(q)))\cap\{y_1\}\times V$. So for any open set $U\times V$ there exists a point $(y_1,y_2)$ in $(\cup_{n\in\mathbb{N}}(f,g)^{-n}((p,q)))\cap U\times V$.
\end{proof}

For example for the product of a doubling map on $S^1$ and the map $B$ in the example \ref{xmp for linear a. endo.}, the set of pre-images of any point in $\mathbb{T}^3$ is dense in the manifold. In this way wee can define a collection of examples and non-examples by defining product spaces of expanding maps on $S^1$ and Anosov diffeomorphisms and endomorphisms on arbitrary manifolds.

But what can be said about non-transitive Anosov endomorphisms? If we consider an endomorphism $f:M\rightarrow M$, according to Lemma \ref{densityofpreimagesofadense} and proposition \ref{densityofpreimagesofadenseorbit}, first we should find subsets of $M$ in which $f$ is transitive.
In this matter by considering just forward orbits of points in $M$, we have Smale and Bowen's spectral decomposition theorem that is introduced for hyperbolic endomorphisms by Sakai. There are subsets in which there are points that their orbit is dense in those sets.\\
Denote the non-wandering set of $f$ by $\Omega$, we have;

\begin{trm}[Smale-Bowen Spectral Decomposition Theorem]
\cite{sakai} Let $f:M\rightarrow M$ be an endomorphism. $f(\Omega)=\Omega$ and $f:\Omega\rightarrow\Omega$ is an Anosov endomorphism, there is a decomposition of $\Omega$ into disjoint closed sets $P_1\cup P_2\dots\cup P_s$ such that;
\begin{itemize}
\item Each $P_i$ is $f-$invariant and $f$ restricted to $P_i$ is topologically transitive.
\item There is a decomposition of each $P_i$ into disjoint closed sets $X_{1,i}\cup X_{2,i}\cup\dots\cup X_{n_i,i}$ such that $f(X_j,i)=f(X_{j+1},i)$, for $1\leq j\leq n+1$, $f(X_{n_i,i})=(X_{1,i})$ and the map $f^{n_i}:X_{j,i}\rightarrow X_{j,i}$ is topologically mixing.  
\end{itemize} 
\end{trm}
$P_i$s ($i=1,2,\dots,s$), introduced above, are called \emph{basic sets} of $f$.

If the degree of $f$ is $k$ then there are $k$ pre-images of each $P_i$ and for every point $p\in P_i$ its set of pre-images is a subset of $\cup_{n\in\mathbb{N}} f^{-n}(P_i)$ where by the notion in remark \ref{sheets of a cover}, $f^{-n}(P_i)=\cup_{j_1,\dots,j_k}P_i(j_1,\dots,j_k)$.

If there are more than one basic sets then by considering $f|_{P_i}$s, according to Lemma \ref{densityofpreimagesofadense} and proposition \ref{densityofpreimagesofadenseorbit}, the set of points with dense set of pre-images is dense in the set of pre-images of $P_i$s $\cup_{n\in\mathbb{N}} f^{-n}(P_i)$ where $i=1,2,\dots,s$. But the set of pre-images of $P_i$ cannot be dense in $M$ because $P_i$s are $f$-invariant; if $x\in P_i$ then $O_x\in P_i$ then if $f^{-1}(x)\nsubseteq f^{-1}(P_i) $ then there is $y\in f^{-1}(x)\cap P_j$ $(j\neq i)$, so $x=f(y)\in P_j$ which is a contradiction and we have;

\begin{prp}
Let $f:M\rightarrow M$ be a hyperbolic endomorphism such that $\Omega(f)=P_1\cup P_2\cup\dots\cup P_s$ $(s>1)$ there are not any points with dense set of pre-images in $M$.
\end{prp}

Thus according to theorem \ref{resultfortransitiveendos}, corollary \ref{product of them} and the proposition above we have the proof of theorem \ref{theresultiff}, our main theorem.

\section{Appendix}
Using the program MATLAB, here we have calculated and demonstrated the pre-images of the point $(0,0)\in[0,1]\times[0,1]$, respectively under $B^{5}$, $B^{10}$ and $B^{15}$ of the linear endomorphism $B=\begin{bmatrix}
3&1\\
1&1
\end{bmatrix}$ in the Example \ref{xmp for linear a. endo.}. it shows that for each $\epsilon$ there is $n$ such that $B^{-n}((0,0))$ is $\epsilon$-dense in $[0,1]\times[0,1]$. Hence the set containing all the pre-images of the point, is dense in $\mathbb{T}^2$.
\begin{center}
\includegraphics[scale=.32]{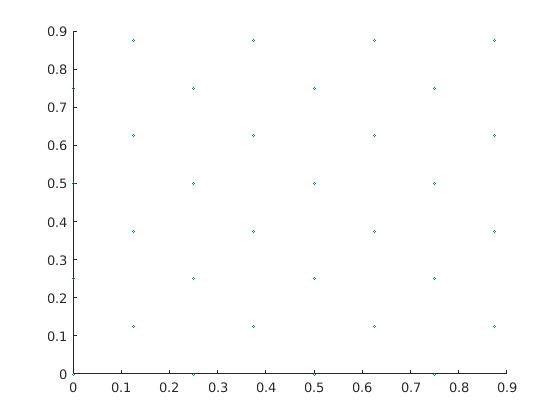}
\includegraphics[scale=.32]{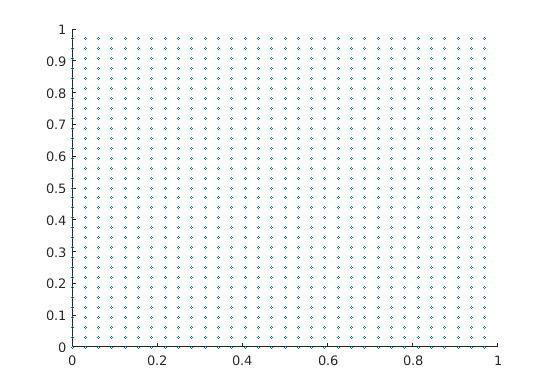}
\includegraphics[scale=.32]{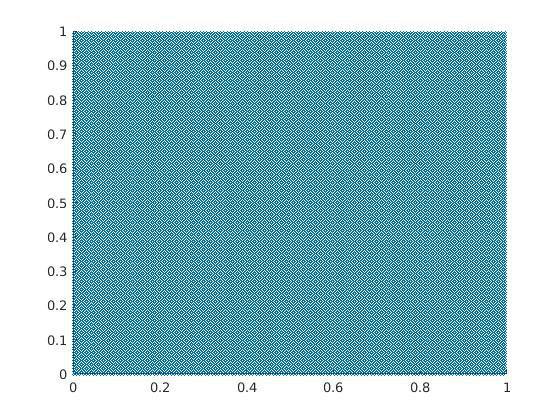}
\end{center}


\begin{thebibliography}{9}

\bibitem{stock} M. Brin, G. Stock, \emph{Introduction to Dynamical Systems}, Cambridge university press , (2003).

\bibitem{Franks} J. Franks, \emph{Anosov Diffeomorphisms, in Global Analysis}, (Proc. Sympos. Pure Math., Vol. 14,
Berkeley, Calif., 1968), Amer. Math. Soc., Providence, R.I., (1970), 61--93.

\bibitem{LPV} C. Lizana, V. Pinheiro, P. Varandas, \emph{Contribution to the Ergodic Theory of Robustly Transitive Maps}, Disc. \& Cont. Dynam. Sys., Vol. 35, No. 1, (2015).


\bibitem{lizana} C. Lizana, E. Pujalz, \emph{Robust Transitivity for Endomorphisms}, Ergod. Th. \& Dynam. Sys., (2013), 1082--1114.

\bibitem{manepugh} R. Ma\~{n}\'{e}, C. Pugh,\emph{Stability of Endomorphisms}, Warwick Dynamical Systems, (1974), 175--184.

\bibitem{MT}  F. Micena, A. Tahzibi, \emph{On the Unstable Directions and Lyupanov Exponents of Anosov Endomorphisms},  Fund. Math. 235 (2016), no. 1, 37--48.

\bibitem{przytycki}  F. Przytycki, \emph{Anosov Endomorphisms}, Studia Mathematica, (1976), 249--285.

\bibitem{sakai} K. Sakai, \emph{Anosov Maps on Closed Topological Manifolds},  J. Math. Soc. Japan 39 (1987), no. 3, 505--519.

\bibitem{Shub} M. Shub, \emph{Global Stability of Dynamical Systems} Springer-Verlag, (1987).

\bibitem{Spatzier} R. Spatzier, \emph{On the Work of Rodriguez Hertz on Rigidity in Dynamical Systems} Journal of Modern Dynamics 10, (2016), 191--207.


\bibitem{W} L. Wen, \emph{Differentiable Dynamical Systems. An Introduction to Structural Stability and Hyperbolicity}Graduate Studies in Mathematics, 173. American Mathematical Society, (2016).
\end{thebibliography}
\end{document}